\documentclass{svproc}

\usepackage{graphicx}              

\usepackage{siunitx}
\usepackage{amsbsy}
\usepackage{amsmath,bm}
\usepackage{amssymb}
\usepackage{amsfonts}
\usepackage{tikz}

\usepackage{multirow}

\usepackage[bottom]{footmisc}
\usepackage{cite}

\usepackage{url}

\begin{document}
\mainmatter

\title{A magnetic oriented approach to the systematic coupling of field and circuit equations}
\titlerunning{Magnetic oriented approach to field-circuit coupling}
\author{Herbert Egger\inst{1}
\and
Idoia Cortes Garcia\inst{2}
\and
Vsevolod Shashkov\inst{3} 
\and 
Michael Wiesheu\inst{4}
}
\institute{Johann Radon Institute for Computational and Applied Mathematics and Johannes Kepler University, Linz, Austria,
\email{herbert.egger@jku.at}
\and 
Department of Mechanical Engineering-Dynamics and Control, Eindhoven University of Technology, Eindhoven, The Netherlands,
\email{i.cortes.garcia@tue.nl}
\and  
Department of Mathematics, Technical University of Darmstadt, Germany,
\email{shashkov@mathematik.tu-darmstadt.de}
\and 
Institute for Accelerator Science and Electromagnetic Fields, Technical University of Darmstadt, Germany, 
\email{michael.wiesheu@tu-darmstadt.de}}

%
%
\maketitle

\begin{abstract}
A novel strategy is proposed for the coupling of field and circuit equations when modeling power devices in the low-frequency regime. The resulting systems of differential-algebraic equations have a particular geometric structure which explicitly encodes the energy storage, dissipation, and transfer mechanisms. This implies a power balance on the continuous level which can be preserved under appropriate discretization in space and time. The models and main results are presented in detail for linear constitutive models, but the extension to nonlinear elements and more general coupling mechanisms is possible. The theoretical findings are demonstrated by numerical results.
\end{abstract}

\section{Introduction}
\label{shashkov:sec:intro}

The design and optimization of electromagnetic devices requires the accurate, efficient, and reliable simulation of the dynamic behavior of magnetic fields in the components and the electric networks used for excitation or interconnection.
Electric machines and power transformers operate in the low frequency regime such that the equations of magneto-quasistatics are suitable as mathematical models~\cite{rodriguez2010eddy} for these devices.
Finite dimensional systems can be obtained by appropriate discretization \cite{kameari1990,meunier2008}.
Suitable models for the excitation through supplied voltages or currents have been considered in~\cite{dular1999circuit,hiptmair2005current}. 
Various methods are also available for modeling electric supply circuits, in particular, the modified nodal and loop analyses~\cite{guenther2005circuits}. 
Based on combinations of the mentioned approaches, the systematic coupling of field and circuit models has been investigated thoroughly; see e.g. \cite{tsukerman1992,degersem2004coupled,degersem2006transient,meunier2008coupling,demenko2010}. 
Further topics that have been addressed are the well-posedness and differential-algebraic index of field-circuit coupled systems~\cite{bartel2011structural,garcia2020structural} or the extension to nonlinear problems in frequency domain~\cite{escarela2010nonlinear,sabariego2018eddy}.

\medskip 
\noindent 
\textbf{Scope and main contributions.}
%
Instead of the conventional nodal or loop analysis, we consider the coupling of a vector potential formulation for the magneto-quasistatic field equations with the magnetic oriented nodal analysis, introduced recently in~\cite{shashkov2022mona}. 
It turns out that both models share the same geometric structure which is preserved when coupling them appropriately. 
Vice versa, the separate models for the magnetic device and the electric circuit can be obtained as special cases by reduction of the coupled system. 
As shown in \cite{egger2021energy}, the underlying structure of the model can  be preserved under appropriate discretization in space and time, such that higher-order power-preserving time discretization and reduced order models can be obtained in a systematic manner. 
Furthermore, compared to the conventional circuit models, the differential-algebraic index is reduced in most cases, which promotes numerical stability. 
In this paper, we introduce the relevant equations for field-circuit coupling, we comment on the underlying problem structure, and we derive the power balance for the resulting systems. For ease of presentation, we restrict ourselves to linear constitutive laws and to the coupling by stranded conductors. 
The generalization to nonlinear models and solid conductors is, however, possible; see~\cite{shashkov2022mona,shashkov2024} for a comprehensive discussion. 

\medskip 
\noindent 
\textbf{Outline.}
%
In Section~\ref{shashkov:sec:model} 
we introduce the mathematical models for magnetic devices and electric circuits, and we briefly discuss their similarities. 
The coupling strategy is introduced in Section~\ref{shashkov:sec:coupling}, where we state and prove the main result of this note and also discuss some possible extensions. 
In Section~\ref{shashkov:sec:num}, we report numerical tests for a typical benchmark problem, which serve for illustration of our theoretical results. 
Section~\ref{shashkov:sec:disc} summarizes our findings and highlights some topics of ongoing  research.

\section{Magnetic oriented modeling}
\label{shashkov:sec:model}

We start with introducing the independent models for the electric circuit and the magnetic device. As mentioned above, we restrict the presentation to linear constitutive equations. The extension to nonlinear models is discussed in \cite{shashkov2024}.

\subsubsection*{Magnetic device model.}

We are interested in low-frequency applications for which the eddy-current approximation of Maxwell's equations is appropriate~\cite{rodriguez2010eddy}. 
The quasistatic field distribution is described by the $A^*$-formulation \cite{kameari1990}, i.e.,
\begin{align} \label{shashkov:eq:1}
\sigma \partial_t \mathbf{a} + \operatorname{curl}(\nu \operatorname{curl} \mathbf{a}) &= \mathbf{j},
\end{align}
complemented by appropriate boundary and gauging conditions. The modified vector potential $\mathbf{a}$ is chosen here, such that $\mathbf{e} = -\partial_t \mathbf{a}$ and $\mathbf{b} = \operatorname{curl} \mathbf{a}$ are the electric field and magnetic flux density; this amounts to a specific gauge~\cite{kameari1990}.
Following \cite{dular1999circuit,schoeps2013winding}, we represent the current density $\mathbf{j} = \boldsymbol{\chi} i_M$ in the stranded conductors via appropriate winding functions $\boldsymbol{\chi}$ and a vector $i_M$ of driving currents supplied by the circuit. 
After discretization by finite elements~\cite{meunier2008}, one obtains the finite-dimensional system 
\begin{align} \label{shashkov:eq:2}
M_\sigma \partial_t a + K_\nu a &= X i_M.
\end{align}
The terminal voltages $v_M = X^\top \partial_t a$ can be interpreted as the output of the system.
Any solution of the above problem satisfies the power balance
\begin{align} \label{shashkov:eq:3}
\frac{d}{dt} H(a) 
&= -\|\partial_t a\|_{M_\sigma}^2 +  \langle i_M, v_M\rangle,
\end{align}
where $H(a) = \frac{1}{2} \|a\|_{K_\nu}^2$ denotes the energy stored in the system. 
For ease of notation, we abbreviate $\|x\|_M^2 := x^\top M x$ and denote by $\langle x, y \rangle = y^\top x$ the Euclidean scalar product.
The above power balance is a special case of the one proven in Theorem~\ref{shashkov:thm:1} below, and states that the system energy can only change due to eddy current losses and power supplied to or drawn from the system.

\subsubsection*{Electric circuit model.}
We consider a circuit consisting of resistors, capacitors, inductors, and voltage resp. current sources.
The interconnection is modeled by a finite directed graph. Using standard notation~\cite{bartel2011structural,guenther2005circuits}, we write $A_*$ for the reduced branch-to-node incidence matrices of the elements~$* \in \{R,C,L,V,I\}$. 
For our further analysis, we introduce magnetic node potentials~$\psi$, branch fluxes~$\phi_*$, and charges~$q_*$, such that 
\begin{align} \label{shashkov:eq:4}
u = \partial_t \psi, \qquad v_* = \partial_t \phi_* , 
\qquad \text{and} \qquad  
i_* = \partial_t q_*
\end{align}
are the vectors of electric node potentials, branch voltages, and branch currents of the conventional nodal analysis~\cite{guenther2005circuits}. 
The magnetic potentials and fluxes are further related by $\phi_* = A_*^\top \psi$, encoding Kirchoff's voltage law. 
We use linear constitutive models $i_R = G v_R$, $i_C = C \partial_t v_C$, and $\partial_t i_L = L^{-1} v_L$, and assume independent voltage and current sources. 
The physical behavior of the circuit under consideration is then described by the system 
\begin{alignat}{10}
A_R G A_R^\top \partial_t \psi + A_C \partial_t q_C + A_V \partial_t q_V + A_L L^{-1} A_L^\top \psi &= -A_I i_{src}\label{shashkov:eq:5}\\
-A_C^\top \partial_t \psi  + C^{-1} q_C &= 0 \label{shashkov:eq:6}\\
-A_V^\top \partial_t \psi  &= -v_{src}. \label{shashkov:eq:7}
\end{alignat}
The first line is Kirchoff's current law, the second represents the model for the capacitors, and the last encodes the effect of the voltage sources. 

Let us denote by $H(\psi,q_C) = \frac{1}{2} \|A_L^\top \psi\|_{L^{-1}}^2 + \frac{1}{2}\|q_C\|_{C^{-1}}^2$ the energy stored in inductors and capacitors. 
Then from the particular structure of the circuit equations, one can see that any solution satisfies the power balance 
\begin{align} \label{shashkov:eq:8}
\frac{d}{dt} H(\psi,q_C) 
&= -\|A_R^\top \partial_t \psi\|^2_{G}  
- \langle i_{src}, v_I\rangle
- \langle i_V, v_{src}\rangle.
\end{align}
For convenience of notation, we introduced additional symbols $v_I=A_I^\top \partial_t \psi$ and $i_V = \partial_t q_V$ for the terminal voltages and currents. The above statement is again  a special case of the power balance stated in Theorem~\ref{shashkov:thm:1} below.
\begin{remark}
For a circuit consisting of resistors, inductors, and current sources only, we obtain the reduced system 
\begin{align} \label{shashkov:eq:9}
A_R G A_R^\top \partial_t \psi + A_L L^{-1} A_L \psi &= - A_I i_{src}, 
\end{align}
which has the same structure as~\eqref{shashkov:eq:2}. Hence \eqref{shashkov:eq:9} can be interpreted as a reduced model for magneto-quasistatics; see \cite{kettunen2001fields,degersem2004coupled} for similar considerations.
\end{remark}

\section{Field-circuit coupling}
\label{shashkov:sec:coupling}

In order to integrate the magnetic device into the circuit, we simply express the driving currents in \eqref{shashkov:eq:2} by $i_M = \partial_t q_M$ and add the output voltages $v_M = X^\top \partial_t a$ as voltage sources to the circuit equations~\eqref{shashkov:eq:5}--\eqref{shashkov:eq:7}. The additional incidence matrix $A_M$ ise used to describe the interconnection of the magnetic device terminals within the circuit. 
This leads to the system
\begin{alignat}{10}
A_R G A_R^\top \partial_t \psi + A_C \partial_t q_C + A_V \partial_t q_V + A_M \partial_t q_M + A_L L^{-1} A_L^\top \psi &=-A_I i_{src}\label{shashkov:eq:10} \\
-A_C^\top \partial_t \psi + C^{-1} q_C &= 0 \label{shashkov:eq:11} \\
-A_V^\top \partial_t \psi  &= -v_{src} \label{shashkov:eq:12} \\
-A_M^\top \partial_t \psi + X^\top \partial_t a &= 0 \label{shashkov:eq:13} \\
M_\sigma \partial_t a  + K_\nu a - X \partial_t q_M &= 0. \label{shashkov:eq:14}
\end{alignat}
For clarity, all quantities to be determined were collected on the left hand side of the equations. The independent voltage and current sources $v_{src}$, $i_{src}$ appear as inputs on the right hand side of the problem. 
We are now in the position to state and prove the main result of this paper. 
\begin{theorem}[Power-balance of the coupled system] \label{shashkov:thm:1} $ $ \\
Let $(\psi,q_C,q_V,q_M,a)$ be a solution of \eqref{shashkov:eq:10}--\eqref{shashkov:eq:14} and set $v_I=A_I^\top \partial_t \psi$, $i_V = \partial_t q_V$. 
Further define $H(\psi,q_C,a)=\frac{1}{2}(\|A_L^\top \psi\|^2_{L^{-1}} + \|q_C\|_{C^{-1}}^2 + \|a\|_{K_{\nu}}^2)$. 
Then
\begin{align} \label{shashkov:eq:15}
\frac{d}{dt} H(\psi,q_C,a) 
= -\|A_R^\top \partial_t \psi\|_{G}^2
- \|\partial_t a\|_{M_\sigma}^2 
- \langle v_I, i_{src}\rangle
- \langle v_{src}, i_V\rangle.
\end{align}
The total energy stored in the system thus changes only by dissipation due to resistive losses and by power supplied to or drawn from the system.
\end{theorem}
\begin{proof}
By formal differentiation of $H(\psi,q_c,a)$ with respect to time, we get
\begin{align*}
\frac{d}{dt} H&(\psi,q_C,a) 
= \langle A_L^\top L^{-1} A_L \psi, \partial_t \psi\rangle + \langle C^{-1}  q_C, \partial_t q_C\rangle + \langle K_\nu a, \partial_t a\rangle \\
 =&-\langle A_R G A_R^\top \partial_t \psi, \partial_t \psi \rangle - \langle A_C \partial_t q_C, \partial_t \psi \rangle - \langle A_V \partial_t q_V, \partial_t \psi \rangle - \langle A_M \partial_t q_M, \partial_t \psi \rangle \\
& 
- \langle A_I i_{src}, \partial_t \psi\rangle
+ \langle A_C^\top \partial_t \psi, \partial_t q_C\rangle 
 - \langle M_\sigma \partial_t a, \partial_t a \rangle + \langle X\partial_t q_M, \partial_t a \rangle. 
\end{align*}
The second identity was obtained by testing \eqref{shashkov:eq:10}, \eqref{shashkov:eq:11}, and \eqref{shashkov:eq:14} with $\partial_t \psi$, $\partial_t q_C$, and $\partial_t a$, respectively. 
Using that $\langle B x, y\rangle = \langle B^\top y, x \rangle$, one can further see that the second and sixth term on the right hand side cancel each other. 
The first and seventh term can be expressed as $-\|A_R^\top \partial_t \psi\|_G^2$ and $-\|\partial_t a\|_{M_\sigma}^2$, which already appear in the final power balance \eqref{shashkov:eq:15}. 
The fifth term can be written as
\begin{align*}
-\langle A_I i_{src}, \partial_t \psi\rangle
=-\langle A_I^\top \partial_t \psi , i_{src}\rangle
=-\langle v_I , i_{src}\rangle,
\end{align*}
which amounts to the first source term in \eqref{shashkov:eq:15}.
For the third term, we use 
\begin{align*}
-\langle A_V \partial_t q_V, \partial_t \psi \rangle
= -\langle A_V^\top \partial_t \psi, \partial_t q_V \rangle
= -\langle v_{src}, \partial_t q_V\rangle = -\langle v_{src}, i_V\rangle
\end{align*}
which follows by testing \eqref{shashkov:eq:12} with $\partial_t q_V$. 
This yields the second source term in~\eqref{shashkov:eq:15}.
The fourth term can finally be rewritten as 
\begin{align*}
-\langle A_M \partial_t q_M, \partial_t \psi \rangle
= -\langle A_M^\top \partial_t \psi, \partial_t q_M \rangle
= -\langle X^\top \partial_t a, \partial_t q_M\rangle
= -\langle X\partial_t q_M, \partial_t a\rangle,
\end{align*}
which follows by testing \eqref{shashkov:eq:13} with $\partial_t q_M$. 
Thus this term cancels out with the last term on the right hand side of the power balance above, 
which leads to the result of the theorem. 
\hfill \qed
\end{proof}

\begin{remark}
The power balances for the separate subsystems, e.g., the magnetic device and the electric circuit,  considered in the previous section are particular cases that follow by dropping some of the terms in the proof of Theorem~\ref{shashkov:thm:1}. 
\end{remark}

\subsubsection*{Abstract problem structure and consequences.}

For the following considerations, let us denote by $y = (\psi,q_C,q_V,q_M,a)$ the vector containing all state variables and by 
$H(y)=H(\psi,q_C,a)$ the associated energy functional. 
The coupled system \eqref{shashkov:eq:10}--\eqref{shashkov:eq:14} can then be written in compact form
\begin{align}
 C(\partial_t y) 
+ \partial_y H(y) &= f(y), 
\end{align}
with source vector $f(y)=(0,0,-v_{src},0,0)$ and $\partial_y H(y)$ denoting the gradient of the energy functional.
The term $C(\partial_t y)$ represents the dissipation and coupling between the individual components, and is monotone, i.e., $\langle C(\partial_t y),\partial_t y\rangle \ge 0$. 
The power balance of the previous theorem can then be stated compactly as 
\begin{align}
\frac{d}{dt} H(y) 
&= \langle \partial_y H(y), \partial_t y\rangle 
= -\langle C(\partial_t y), \partial_t y\rangle + \langle f(y), \partial_t y\rangle 
\end{align}
with $\langle C(\partial_t y), \partial_t y\rangle$ and $\langle f(y), \partial_t y\rangle$ describing the dissipated and the supplied power, respectively. 
This compact form of the problem clearly reveals the underlying geometric structure of the formulation which offers the possibility for various extensions and corollaries, namely:
\begin{itemize}
\item the considerations of nonlinear energy-based constitutive models;
\item the power-preserving time-discretization through Petrov-Galerkin methods;
\item the systematic model order reduction by Galerkin projection.
\end{itemize}
Details and proofs for corresponding generalizations can be found in \cite{egger2021energy,shashkov2022mona,shashkov2024}.

\section{Numerical results}\label{shashkov:sec:num}

For an illustration of our theoretical findings, we consider a simple model for a full wave rectifier. 
The electric circuit consists of a voltage source~(V), a transformer~(M), four diodes (D$_i$), and a load resistance~(R); see Figure~\ref{shashkov:fig:rectifier} for a schematic sketch of the geometric setup.
\begin{figure}[ht!]
\includegraphics[scale= 1.06]{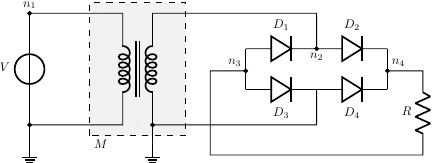}
\hfill
\includegraphics[scale= 0.9]{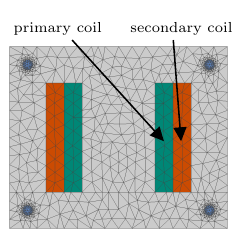}
\caption{Schematic sketch of a full wave rectifier circuit (left) and geometry of transformer modeled by the field equations (right). Based on \cite[Section~6.2]{schops2011multiscale}.}
\label{shashkov:fig:rectifier}
\end{figure}
The diodes are modeled as nonlinear resistors with given voltage-current relation $i_D = I_s(\exp(v_D/V_{th}) -1) + v_D/R_{par}$, where $I_s = 10^{-14}\, \unit{\ampere}$, $V_{th} = 2.5\cdot 10^{-2} \, \unit{\volt}$, and $R_{par} = 10^{12} \, \unit{\ohm}$. 
The load resistance is set to $R=10 \, \unit{\ohm}$ and the source voltage is specified as $v_{src}(t) = 160\sin(120\pi t) \, \unit{\volt}$. 
The transformer is described by a linear eddy current model taken from the FEMM example in \cite{meeker2020finite}.
We note that the extension of Theorem~\ref{shashkov:thm:1} and its proof to the problem under consideration, including nonlinear resistors, is possible by minor modifications of the arguments. 
%

\subsubsection*{Simulation setup.}
The space discretization of the transformer is done by piecewise linear finite elements in two dimensions~\cite{meunier2008}. 
For the time discretization of the system, we employ the midpoint rule with uniform step size $\tau>0$. 
In compact notation, the $n$th time step of the resulting scheme has the form
\begin{align*}
C(d_\tau y^{n-1/2}) 
+ \partial_y  H(y^{n-1/2}) = f(t^{n-1/2}) \qquad n \ge 0,
\end{align*}
with $d_\tau y^{n-1/2} = \frac{1}{\tau}(y^{n}-y^{n-1})$, $y^{n-1/2}=\frac{1}{2}(y^{n}+y^{n-1})$, and $t^{n-1/2} = \frac{1}{2}(t^{n} + t^{n-1})$.  
The nonlinear systems in every time-step are solved by the Newton method with a tolerance of $10^{-12}$. 
Finally, the initial values of all variables are chosen zero, which is consistent with the algebraic constraints of the system. 
%

\subsubsection*{Discrete power balance.}
For our test problem, the transformer is the only energy storing element. Since a linear material model is used for this device, the energy of the system is a quadratic function of the state variables.
As a consequence, a discrete power balance can be established for our scheme, namely
\begin{align*}
    d_\tau H^{n-1/2}  &= 
    - \|d_\tau a^{n-1/2}\|^2_{M_\sigma} -  d(A_R^\top d_\tau \psi^{n-1/2}) - \langle v_{src}(t^{n-1/2}), i_V^{n-1/2}\rangle.
\end{align*}
Here $H^n = \tfrac{1}{2}\|a^n\|^2_{K_\nu}$ is the stored energy at time $t^n$,  $\|d_\tau a^{n-1/2}\|^2_{M_\sigma}$ are the eddy current losses, $d(A_R^\top d_\tau \psi^{n-1/2}) \ge 0$ represents the dissipation in load resistor and diodes, and $i_V^{n-1/2} = d_\tau q_V^{n-1/2}$ denotes the current across the voltage source. 

\subsubsection*{Numerical results.}
In Figure~\ref{shashkov:fig:results}, we display the input voltage $v_{src}$ and the rectified output voltage  $v_{R}$ as a functions of time. In addition, we list the absolute maximum discretization errors ($\epsilon_\tau$), the observed convergence rates (e.o.c.), and the maximum discrepancy in the discrete power balance ($\epsilon_H$). 
\begin{figure}[ht]
\begin{minipage}{.55\textwidth}
\includegraphics[scale=0.8]{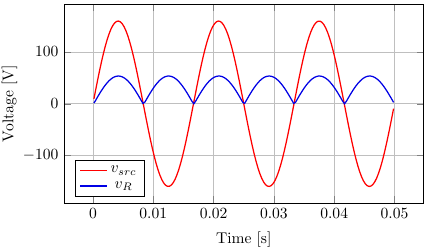}
\end{minipage}
\hspace*{-2.5em}
\begin{tabular}{c||c|c||c}
 $\tau$  \, &  \, $\epsilon_\tau $ \,   &  e.o.c. & \, $\epsilon_H$ \,  \\
\hline
0.005000 & 0.0266 & --   &  0.0853 $\cdot 10^{-12}$ \\
0.002500 & 0.0078 & 1.77 &  0.1137 $\cdot 10^{-12}$\\
0.001250 & 0.0021 & 1.89 &  0.1208 $\cdot 10^{-12}$\\
0.000625 & 0.0006 & 1.79 &  0.0782 $\cdot 10^{-12}$
\end{tabular}
\caption{Left: input voltage $v_{src}$ and rectified output voltage $v_R$. Right: convergence history for the magnetic potential $\psi_4$ and maximal discrepancy in the energy balance.}
\label{shashkov:fig:results}
\end{figure}
%
%
The computational results clearly illustrate the desired physical behavior. 
As expected for the mid-point rule, we observe second order convergence in the primary unknowns. Moreover, the discrete power balance is satisfied at each time-step with the same level of accuracy as used for the solution of the nonlinear systems. 
%

%
%

\section{Discussion} \label{shashkov:sec:disc}

A novel magnetic-oriented approach to the coupling of electric circuits and magnetic devices was proposed. The underlying power balance of the system is encoded in the particular geometric structure of the resulting equations, which can be preserved by appropriate discretization strategies. 
The approach was presented for the case of linear device models but can be extended systematically to the nonlinear case. 
Details about these generalizations can be found in \cite{shashkov2024}. 
Based on the considerations in \cite{shashkov2022mona}, we expect that the differential-algebraic index of the coupled system can be shown to be $\nu \le 1$. A proof of this statement is left for future research. 

\subsubsection*{Acknowledgments.}
The authors are greatful for financial support by the German Science Foundation (DFG) and the Austrian Science Fund (FWF) via grants TRR~361 and SFB~F90-N, subprojects C01 and C02.


\end{document}